\documentclass[11pt]{article}
\usepackage{hyperref}
\usepackage{amsmath,amssymb}
\usepackage{amsthm}
\usepackage{mathrsfs,calrsfs}
\usepackage{mathtools}
\usepackage{epsf}
\usepackage{indentfirst}
\usepackage{latexsym}
\usepackage[dvips]{graphicx}
\usepackage{flafter}
\usepackage{caption}
\usepackage{textcomp}
\usepackage{hhline}
\usepackage{bigstrut,bigdelim}
\usepackage{rotating}
\usepackage{multicol,multirow}
\usepackage{makeidx}
\usepackage{epsfig}
\usepackage{fancyhdr}
\usepackage{booktabs}
\usepackage{mathrsfs}
\usepackage[all]{xy}
\usepackage{xcolor}
\usepackage{microtype}
\usepackage[numbers,sort&compress]{natbib}
\newif\ifger
\gerfalse
\newtheorem{definition}{Definition}[section]
\newtheorem{theorem}[definition]{Theorem}
\newtheorem{lemma}[definition]{Lemma}
\newtheorem{corollary}[definition]{Corollary}
\newtheorem{remark}{Remark}[section]

\newtheorem*{funding}{Funding}
\newtheorem*{availability of data and material}{Availability of Data and Material}
\newtheorem*{conflict of interest}{Conflict of interest}
\newcommand{\Tr}{{\rm Tr}}

\begin{document}
\baselineskip=19pt

\title{Permutation Polynomials of the form $L(X)+\gamma\Tr_q^{q^3}(h(X))$ \\ over finite fields with even characteristic}

\author{ Xuan Pang$^a$, Danyao Wu$^b$, Pingzhi Yuan$^a${\footnote{Corresponding author}} \\
\small \it  $^a$School of Mathematical Sciences, South China Normal University, \\
\small \it  Guangzhou, 510631, P. R. China\\
\small \it  $^b$School of Computer Science and Technology, Dongguan University of Technology,\\
\small \it   Dongguan 523808, P. R. China \\
}

\date{}
\maketitle
\renewcommand{\thefootnote}{}\footnotetext{E-mail addresses: pangxuan202503@163.com (X. Pang),  wudanyao111@163.com (D. Wu), yuanpz@scnu.edu.cn ( P. Yuan).}

\begin{abstract}
Permutation polynomials over finite fields have extensive applications in various areas. Particularly, permutation polynomials with simple forms are of great interest. In recent papers, several classes of permutation polynomials of the form $L(X)+\Tr_q^{q^3}(h(X))$ have been constructed. This paper further investigates permutation polynomials of such form over $\mathbb{F}_{q^3}$. Unlike previous studies, we transform the problem of constructing univariate permutation polynomials over finite fields into that of constructing corresponding multivariate permutations over $\mathbb{F}_{q}$-vector spaces. Through this approach, we completely characterize a class of permutation polynomials of the form $L(X)+\gamma\Tr_q^{q^3}(c_1X+c_2X^2+c_3X^3+c_4X^{q+2})$ over $\mathbb{F}_{q^3}$, where $q=2^m$, $L(X)=X^q+aX$ and $a,c_1,c_2,c_3,c_4,\gamma\in\mathbb{F}_q$ with $a^2+a+1\neq0$. Furthermore, using a similar method, we generalize several results from a recent work by Jiang, Li and Qu (2026).

\medskip
\noindent{\bf MSC(2020):} 11T06; 11T55.

\medskip
\noindent{\bf Keywords:} Finite field; trace function; permutation polynomial.
\end{abstract}

\section{Introduction}
Let $\mathbb{F}_q$ be a finite field of $q$ elements, where $q$ is a prime power. Every mapping from a finite field into itself can be uniquely represented by a polynomial in $\mathbb{F}_q[X]$ of degree less than $q$. A polynomial  $f(X)\in\mathbb{F}_q[X]$ is called a {\em permutation polynomial} (PP for short) if its induced mapping $f:c\mapsto f(c)$ is a bijection on $\mathbb{F}_q$.  PPs are fundamental objects in finite field theory and have far-reaching applications, including in coding theory \cite{Ding,Ding-Zhou,Laigle-Chapuy}, cryptography \cite{Rivest-Shamir-Adelman,Schwenk-Huber,Singh2009,Singh2020}, combinatorial design theory \cite{Ding-Yuan} and so on. For instance, codes constructed from PPs often have relatively high minimum distance, which endows them with powerful error-correcting capabilities and thus enables reliable communication over noisy channels. It is these properties that continue to attract significant research interest in the theory of PPs over the past few decades.

For a prime power $q$ and a positive integer $n$, define the {\em (relative) trace function} from $\mathbb{F}_{q^n}$ to $\mathbb{F}_q$ by
\begin{equation*}\Tr_{q}^{q^n}(\alpha)= \alpha+\alpha^q+\cdots+\alpha^{q^{n-1}},~{\rm for}~ \alpha\in\mathbb{F}_{q^n}.\end{equation*}
If $q=p$ is a prime, it is referred to as the {\em absolute trace function}. Owing to  its nature as a linear map onto a subfield, the trace function serves as a fundamental tool for constructing PPs \cite{ChKy08,ChKy09,ChKy10,Kyzieve2016,Wang2024,Wu2025,ZXY2015,Zha2019}.

Finding PPs of simple forms or nice algebraic structures remains an appealing yet challenging problem. Among the various classes of PPs, Dickson polynomials stand out as a particularly important class, whose permutation behavior has been well characterized in  \cite{German1968}. Inspired by this,  Hollmann and Xiang \cite{Hollmann05} constructed a class of PPs related to Dickson polynomials,
taking the form $G(X)+\gamma\Tr_2^{2^n}(h(X))$. Subsequently, Charpin and Kyureghyan \cite{ChKy08} used Boolean functions with linear structures to study such form. By considering the case where $G(X)$ is a PP or a linearized polynomial, they constructed six classes of such PPs over $\mathbb{F}_{2^n}$.  Their work was later extended to finite fields of arbitrary characteristic in \cite{ChKy09} via the same linear structure approach.   Further building on this, they exhibited several more such PP classes in \cite{ChKy10} by choosing $G(X)$ and $H(X)$ to be monomials.  Collectively, these works establish an alternative approach for constructing PPs over finite fields by leveraging connections with functions having a linear structure.

Given the inherent interest in PPs with few terms, Kyureghyan and Zieve \cite{Kyzieve2016} investigated polynomials of the form $X+\gamma\Tr_q^{q^n}(X^k)$ over finite fields. 
 From computer search, they determined all PPs over  $\mathbb{F}_{q^n}$ of such form with $\gamma\in\mathbb{F}_{q^n}^*$, $n>1$, $q$ odd, and $q^n<5000$.  Moreover, they constructed nine new classes of PPs to explain almost all examples. These findings motivated Li et al. \cite{LiKQ2018Cry.} to study the analogous form $cX+\Tr_{q}^{q^n}(X^a)$ in even characteristic, leading to the construction of fifteen new classes of PPs. In parallel, Ma and Ge \cite{MagE2017}, along with Zha, Hu and Zhang \cite{Zha2019}, continued the work of \cite{Kyzieve2016} by studying PPs of the generalized form $X+\gamma\Tr_q^{q^n}(h(X))$. Their work not only yielded several new PPs of this type but also accounted for the remaining computational examples proposed by Kyureghyan and Zieve. More recently, Jiang et al. \cite{LikqQulj2025, LikqQulj2026} further focused on PPs of the form $X+\gamma\Tr_q^{q^n}(h(X))$, where $q=2^m$ and $n\in\{2,3\}$. By choosing functions $h(X)$ with a low $q$-degree,  they constructed some new classes of PPs over finite fields. The methods employed in these works are standard: One reduces the permutation property over $\mathbb{F}_{q^n}$ into verifying permutations over specific subgroups (AGW criterion); the other determines the number of solutions to certain equations over finite fields.

In this paper, we further investigate PPs over $\mathbb{F}_{q^3}$ of the form $L(X) + \gamma\operatorname{Tr}_q^{q^3}(h(X))$ in even characteristic, where $L(X)=X+aX^q$ with $a\in\mathbb{F}_q$. Departing from conventional methods, we implement the novel technique proposed by \cite{NewPPs2}.
Specifically, we transform the problem of constructing univariate PPs $f(X)$ over $\mathbb{F}_{q^3}$ into that of constructing multivariate permutation mappings $F = (f_1(x,y,z),  f_2(x,y,z), f_3(x,y,z))$ over $\mathbb{F}_q^3$ (see Proposition \ref{pang-yuan-wu-guan}). Using this approach, we first discuss the permutation behavior of polynomials of the form  $L(X)+\gamma\Tr_q^{q^3}(c_1X+c_2X^2+c_3X^3+c_4X^{q+2})$ over $\mathbb{F}_{q^3}$, where $q=2^m$, $L(X)=X^q+aX$, and $a,c_1,c_2,c_3,c_4,\gamma\in\mathbb{F}_q$ with $a^2+a+1\neq0$, thereby completely characterizing this class of PPs. The second objective of this paper is to slightly  generalize certain results from \cite{LikqQulj2026}, where Jiang, Li, and Qu studied PPs of the form $X+\gamma\Tr_q^{q^3}(h(X))$ over $\mathbb{F}_{q^3}$ with $q=2^m$ and $\gamma=1$. They obtained several new PPs and conjectured that the condition $\gamma=1$ is necessary and sufficient. Our work extends their study to arbitrary $\gamma\in\mathbb{F}_q$, thus offering a partial affirmative answer to this conjecture.

The remaining part of this paper is organized as follows. In Section 2, we present some basic concepts and relevant preliminary results from finite field theory. Most importantly, we find a specific $\mathbb{F}_q$-basis of $\mathbb{F}_{q^3}$, based on existing results on normal bases, that serves as a foundational tool for our subsequent analysis of PPs. In Section 3, we dedicate to studying polynomials of the form $L(x)+\gamma\Tr_{q}^{q^3}(h(X))$. This section is divided into two subsections according to the shape of $h(X)$.  Finally, Section 4 concludes this paper and provides some further work.

\section{Preliminaries}
For brevity, the following notations are fixed throughout this paper.

$\bullet$ $\Tr_{q}^{q^n}(\cdot)$ is the trace function from $\mathbb{F}_{q^n}$ to $\mathbb{F}_{q}$.

$\bullet$ $\mathbb{F}_{q}^*$ denotes the set of nonzero elements of $\mathbb{F}_{q}$. That is $\mathbb{F}_{q}^*=\mathbb{F}_q\backslash\{0\}$.

$\bullet$ $\mathbb{F}_{q}^n$ denotes the $n$-dimensional vector space of $\mathbb{F}_q$.

\medskip

In what follows, we list several basic properties over finite fields.

\begin{lemma}\cite{Lidl}\label{constant}
Let $f(X)\in\mathbb{F}_{q}[X]$. Then $f(X)$ is a PP of $\mathbb{F}_{q}[X]$ if and only if $cf(dX)+b$ is a PP of  $\mathbb{F}_{q}[X]$ for any given $c,d\in\mathbb{F}_{q}^*$ and $b\in\mathbb{F}_q$.
\end{lemma}

The algebraic structure of $\mathbb{F}_{q^n}$ over $\mathbb{F}_q$ enables the concrete representation of field elements via a chosen basis. The number of distinct bases of $\mathbb{F}_{q^n}$ over $\mathbb{F}_q$ is rather larger (given by the product $\prod_{i=0}^{n-1}(q^n - q^i)$), but there are two special types of bases of particular importance. The first is a polynomial basis $\{1,\alpha,\alpha^2,\dots,\alpha^{n-1}\}$, made up the powers of a defining element $\alpha$ of $\mathbb{F}_{q^n}$ over $\mathbb{F}_q$. The element $\alpha$ is often taken to be a primitive element of $\mathbb{F}_{q^n}$ (see \cite[Theorem 2.10]{Lidl}). Another type of basis is a normal basis defined by a suitable element of $\mathbb{F}_{q^n}$.

\begin{lemma}\cite[Theorem 2.35]{Lidl}\label{normalbasisTheorem}
Let $\mathbb{F}_q$ be a finite field, and let $\mathbb{F}_{q^n}$ be its extension of degree $n$. Then there exists an element $\alpha\in\mathbb{F}_{q^n}^*$ such that the set $\{\alpha,\alpha^q,\dots,\alpha^{q^{n-1}}\}$ forms a basis for $\mathbb{F}_{q^n}$ over $\mathbb{F}_{q}$. Such a basis is called a normal basis, and  the element $\alpha$ is called a normal element.
\end{lemma}

In what follows, we use Lemma \ref{normalbasisTheorem} to  establish a crucial lemma that underpin our main results.

\begin{lemma}\label{basisfromNB}
Let $\mathbb{F}_q$ be a finite field, and let $\mathbb{F}_{q^n}$ be its extension of degree $n$. If $\gcd(n,q)=1$, then there exists an element $\alpha\in\mathbb{F}_{q^n}^*$ satisfying $\Tr_q^{q^n}(\alpha)=0$ such that $\{1,\alpha,\alpha^q,\dots,\alpha^{q^{n-2}}\}$ forms a basis of $\mathbb{F}_{q^n}$ over $\mathbb{F}_q$.
\end{lemma}
\begin{proof}
From Lemma \ref{normalbasisTheorem}, we choose a normal element $\theta\in\mathbb{F}_{q^3}^*$ such that $\{\theta,\theta^q,\cdots,\theta^{q^{n-1}}\}$ forms a normal basis of $\mathbb{F}_{q^n}$ over $\mathbb{F}_q$. Since the trace function is onto, we have $\Tr_q^{q^n}(\theta)=\theta+\theta^q+\cdots+\theta^{q^{n-1}}=c$ for some nonzero $c\in\mathbb{F}_q$. This gives  $\theta^{q^{n-1}}=c-\theta-\theta^q-\cdots-\theta^{q^{n-2}}$. Hence $\{1,\theta,\theta^q\cdots,\theta^{q^{n-2}}\}$ spans $\mathbb{F}_{q^n}$. Consequently, $\{1, \theta, \theta^q,\cdots,\theta^{q^{n-2}}\}$ is necessarily a basis since the dimension of $\mathbb{F}_{q^n}$ over $\mathbb{F}_q$ is $n$. Assume now that $\alpha=\theta-\frac{c}{n}$. Then we have
\begin{equation*}
\begin{pmatrix}
1    & 0 &0 & \dots &0 \\
-c/n & 1 &0 &\dots & 0 \\
-c/n & 0 & 1 &\dots & 0\\
\vdots&\vdots &\vdots &  & \vdots \\
-c/n & 0 & 0& \dots &1\\
\end{pmatrix}
\begin{pmatrix} 1 \\ \theta \\ \theta^q \\ \vdots \\ \theta^{q^{n-2}}
\end{pmatrix}
=
\begin{pmatrix}
1 \\ \alpha \\ \alpha^q \\ \vdots \\ \alpha^{q^{n-2}}
\end{pmatrix}.
\end{equation*}
Since the matrix is clearly invertible over $\mathbb{F}_q$, the set $\{1,\alpha,\alpha^q,\dots,\alpha^{q^{n-2}}\}$ is also a basis of $\mathbb{F}_{q^n}$ over $\mathbb{F}_q$.  Moreover, a straightforward computation shows that $\Tr_q^{q^n}(\alpha)=0$. This completes the proof.
\end{proof}

Setting $n=3$ in Lemma \ref{basisfromNB} yields the following corollary.

\begin{corollary}\label{basis}
Let $\mathbb{F}_q$ be a finite field with $3\nmid q$.
Then for any degree $3$ extension $\mathbb{F}_{q^3}$ over $\mathbb{F}_q$, there exists an element $\alpha\in\mathbb{F}_{q^3}$ satisfying $\Tr_q^{q^3}(\alpha)=0$ such that $\{1,\alpha,\alpha^q\}$ forms a basis of $\mathbb{F}_{q^3}$ over $\mathbb{F}_q$.
\end{corollary}

The finite field $\mathbb{F}_{q^n}$ is clearly isomorphic to $\mathbb{F}_q^n$. Therefore, studying the permutation behavior of a polynomial $f(X)\in\mathbb{F}_{q^n}[X]$ is equivalent to analyzing the behavior of its corresponding multivariate polynomial mapping $F=(f_1(X_1, \dots, X_n), \dots, f_n(X_1, \dots, X_n))$ over $\mathbb{F}_q^n$.

\begin{lemma}\cite{NewPPs2}\label{pang-yuan-wu-guan}
Let $\{\alpha_1, \alpha_2, \dots, \alpha_n\}$ and $\{\beta_1, \dots, \beta_n\}$  be two bases of $\mathbb{F}_{q^n}$ over $\mathbb{F}_q$. For a polynomial $f(X)\in\mathbb{F}_{q^n}[X]$ and $a_i\in \mathbb{F}_q$ for $1\le i\le n$, we let
\begin{equation*}X=(X_1+a_1, \ldots, X_n+a_n)(\alpha_1, \dots, \alpha_n)^T, \end{equation*}
\begin{equation*}f(X)=(f_1(X_1, \dots, X_n), \dots, f_n(X_1, \dots, X_n))(\beta_1, \dots, \beta_n)^T+c,\end{equation*}
where $c\in\mathbb{F}_{q^n}$. Then $f(X)$ is a PP over $\mathbb{F}_{q^n}$ if and only if $(f_1(X_1, \dots, X_n), \dots, f_1(X_1, \dots, X_n))$ is a permutation of $\mathbb{F}_q^n$.
\end{lemma}

\section{PPs of the form $L(X)+\gamma\Tr_q^{q^3}(h(X))$ over $\mathbb{F}_{q^3}$}

In this section, we are devoted to constructing several classes of PPs over $\mathbb{F}_{q^3}$ the form $X+aX^q+\gamma\Tr_q^{q^3}(h(X))$, where $q=2^m$ and  $\gamma\in\mathbb{F}_q$. We begin by deriving the following key lemma from  Corollary \ref{basis} and Lemma \ref{pang-yuan-wu-guan}.

\begin{lemma}\label{lem3-1}
Let $q=2^m$ and $\alpha\in\mathbb{F}_{q^3}^*$ such that $\Tr_q^{q^3}(\alpha)=0$ and $\{1,\alpha,\alpha^q\}$ is a basis of $\mathbb{F}_{q^3}$ over $\mathbb{F}_q$. Consider the polynomial
\begin{equation*}f(X)=X+aX^q+\gamma\Tr_q^{q^3}(h(X)),\end{equation*}
where $a,\gamma\in\mathbb{F}_q$ with $a^2+a+1\neq0$. Let
$X=x+y\alpha+z\alpha^q$ with $x,y,z\in\mathbb{F}_q.$ Then $f(X)$ is a PP over $\mathbb{F}_{q^3}$ if and only if for each $y,z\in\mathbb{F}_q$, \begin{equation*}g(x,y,z)=(a+1)x+\gamma\Tr_q^{q^3}(h(x+y\alpha+z\alpha^q))\end{equation*} is a PP of $\mathbb{F}_{q}$.
\end{lemma}
\begin{proof}
Since $X=x+y\alpha+z\alpha^q$, $\Tr_q^{q^3}(\alpha)=0$ and $\gamma\in\mathbb{F}_q$,  the polynomial
\begin{align*}
f(X)=f(x+y\alpha+z\alpha^q)&=x+y\alpha+z\alpha^q+a(x+y\alpha+z\alpha^q)^q+\gamma\Tr_q^{q^3}(h(x+y\alpha+z\alpha^q))\\
 &=(a+1)x+\gamma\Tr_q^{q^3}\left(h(x+y\alpha+z\alpha^q)\right)+(y+az)\alpha+(z+ay+az)\alpha^q.
\end{align*}
Clearly $\Tr_q^{q^3}(h(x+y\alpha+z\alpha^q))\in\mathbb{F}_q$. We observe that $f_2(x,y,z)=y+az$ and $f_3(x,y,z)=z+ay+az$ are both independent of $x$, and the pair $(f_1,f_2)$ forms a permutation polynomial pair of $\mathbb{F}_{q}^2$ under the condition $a^2+a+1\neq0$. We thus conclude that $(f_1(x,y,z),f_2(x,y,z),f_3(x,y,z))$  permutes $\mathbb{F}_q^3$ if and only if $f_1(x,y,z)$  permutes $\mathbb{F}_q$. Therefore, $f(X)$ is a PP of $\mathbb{F}_{q^3}$ if and only if for every $y, z \in \mathbb{F}_q$, the polynomial $f_1(x,y,z)$, i.e., $g(x,y,z)$ is a PP of $\mathbb{F}_q$.
\end{proof}

As an application of Lemma \ref{lem3-1}, we present some results about the permutation properties for polynomials of the form $f(X)=X+aX^q+\gamma\Tr_q^{q^3}(h(X))$.


\begin{theorem}\label{tm3.2}
Let $q=2^m$ and $a,\gamma\in\mathbb{F}_q$ with $a^2+a+1\neq0$. A polynomial $f(X)=X+aX^q+\gamma\Tr_q^{q^3}(h(X))$ is a PP of $\mathbb{F}_{q^3}$ if and only if $g(X)=X+aX^q+\gamma\Tr_q^{q^3}\left(h(X)+X^{2^i+2^jq}+X^{2^iq+2^j}\right)$ is a PP of $\mathbb{F}_{q^3}$, where $i$ and $j$ are  non-negative integers.
\end{theorem}
\begin{proof} By  Corollary \ref{basis}, there exists an element $\alpha\in\mathbb{F}_{q^3}^*$ satisfying $\Tr_q^{q^3}(\alpha)=0$ such that $\{1,\alpha,\alpha^q\}$ is a basis of $\mathbb{F}_{q^3}$ over $\mathbb{F}_q$.
Let $X=x+y\alpha+z\alpha^q$ with $x,y,z\in\mathbb{F}_q$.  Then
 \begin{align*}
 X^{2^i+2^jq}&=(x+y\alpha+z\alpha^q)^{2^i}(x+y\alpha^q+z\alpha^{q^2})^{2^j}\\
      &=x^{2^i+2^j}+x^{2^i}(y\alpha^q+z\alpha^{q^2})^{2^j}+x^{2^j}(y\alpha+z\alpha^q)^{2^i}+(y\alpha+z\alpha^q)^{2^i}(y\alpha^q+z\alpha^{q^2})^{2^j}.
 \end{align*}
Taking the trace function on both sides of the above expression, we obtain $\Tr_q^{q^3}(X^{2^i+2^jq})=x^{2^i+2^j}+\ell(y,z),$
where $\ell(y,z)=\Tr_q^{q^3}\left((y\alpha+z\alpha^q)^{2^i}(y\alpha^q+z\alpha^{q^2})^{2^j}\right)\in\mathbb{F}_q[y,z]$ is a constant with respect to $x$. By analogy, we have $\Tr_q^{q^3}(X^{2^iq+2^j})=x^{2^i+2^j}+\ell'(y,z)$, where $\ell'(y,z)=\Tr_q^{q^3}\left((y\alpha+z\alpha^q)^{2^j}(y\alpha^q+z\alpha^{q^2})^{2^i}\right)\in\mathbb{F}_q[y,z]$.
Therefore, we deduce that
\begin{align*}
g(X)&=g(x+y\alpha+z\alpha^q)\\
 &=x+y\alpha+z\alpha^q+a(x+y\alpha+z\alpha^q)^q+\gamma\left(\Tr_q^{q^3}(h(x+y\alpha+z\alpha^q))+\ell(y,z)+\ell'(y,z)\right)\\
  &=(a+1)x+\gamma\left(\Tr_q^{q^3}(h(x+y\alpha+z\alpha^q))+\ell(y,z)+\ell'(y,z)\right)+(y+az)\alpha+(z+ay+az)\alpha^q.
\end{align*}
It follows from Lemma \ref{lem3-1} that $g(X)$ is a PP of $\mathbb{F}_{q^3}$ if and only if for each $y,z\in\mathbb{F}_q$,
\begin{equation*}
(a+1)x+\gamma\left(\Tr_q^{q^3}(h(x+y\alpha+z\alpha^q))+\ell(y,z)+\ell'(y,z)\right)
\end{equation*} permutes $\mathbb{F}_q$.  This is equivalent to, for each $y,z\in\mathbb{F}_q$, the polynomial $(a+1)x+\gamma\Tr_q^{q^3}(h(x+y\alpha+z\alpha^q))$ being a permutation of $\mathbb{F}_{q}$, as both $\gamma\ell(y,z)$ and $\gamma\ell'(y,z)$ are independent of $x$. Consequently, $g(X)$ permutes $\mathbb{F}_{q^3}$ if and only if  $(a+1)x+\gamma\Tr_q^{q^3}(h(x+y\alpha+z\alpha^q))$ permutes $\mathbb{F}_{q}$ for any $y,z\in\mathbb{F}_q$. A similar argument applied to $f(X)$ gives the same necessary and sufficient conditions.  We thus complete the proof.
\end{proof}

Observing the result above, we see that if we modify the exponent by replacing $2^i+2^jq$  with $2^i+2^j$ while keeping the rest of the polynomial structure unchanged, a similar equivalence for PPs still holds. This leads to the following theorem:

\begin{theorem} \label{tm3.3}
Let $q=2^m$ and $a,\gamma\in\mathbb{F}_q$ with $a^2+a+1\neq0$. A polynomial $f(X)=X+aX^q+\gamma\Tr_q^{q^3}(h(X))$ is a PP of $\mathbb{F}_{q^3}$ if and only if $g(X)=X+aX^q+\gamma\Tr_q^{q^3}\left(h(X)+X^{2^i+2^j}+X^{2^iq+2^j}\right)$ is a PP of $\mathbb{F}_{q^3}$, where $i$ and $j$ are  non-negative integers.
\end{theorem}

Building on Theorems \ref{tm3.2}-\ref{tm3.3}, we can further generalize the conclusion as follows.
\begin{corollary}
Let $q=2^m$ and $a,\gamma\in\mathbb{F}_q$ with $a^2+a+1\neq0$. A polynomial \begin{equation*}f(X)=X+aX^q+\gamma\Tr_q^{q^3}(h(X))\end{equation*} is a PP of $\mathbb{F}_{q^3}$ if and only if \begin{equation*}g(X)=X+aX^q+\gamma\Tr_q^{q^3}\left(h(X)+\sum_{i,j}c_{ij}\left(X^{2^i+2^j}+X^{2^iq+2^j}\right)+\sum_{s,t}d_{st}\left(X^{2^s+2^tq}+X^{2^sq+2^t}\right)\right)\end{equation*} is a PP of $\mathbb{F}_{q^3}$, where $c_{ij}, d_{st}\in\mathbb{F}_{q^3}$, and $i,j,s,t$ are non-negative integers.
\end{corollary}

\subsection{PPs of the form $L(X)+\gamma\Tr_q^{q^3}(c_1X+c_2X^2+c_3X^3+c_4X^{q+2})$}

In a recent study, Jiang, Li and Qu \cite{LikqQulj2025} investigated  PPs over $\mathbb{F}_{q^2}$ (where $q=2^m$) of the form \begin{equation*}f(X)=X+\gamma\Tr_q^{q^2}(c_1X+c_2X^2+c_3X^3+c_4X^{q+2}),\end{equation*}
where $\gamma\in\mathbb{F}_{q^2}$ and $c_i\in\{0,1\}$ for $1\leq i\leq4$.
Their work, aided by Magma computations, led to the complete classification of six classes of PPs with such a form. Motivated by their work, we carry out an analogous investigation for the cases over $\mathbb{F}_{q^3}$. Specifically, we focus on polynomials of the form
$X+aX^q+\gamma\Tr_q^{q^3}(c_1X+c_2X^2+c_3X^3+c_4X^{q+2})$ over $\mathbb{F}_{q^3}$, where $q=2^m$, and $a,c_i,\gamma\in\mathbb{F}_q$ for $1\leq i\leq4$, with the condition $a^2+a+1\neq0$.


\begin{theorem}
Let $q=2^m$, and let $a,c_i,\gamma\in\mathbb{F}_q$ for $1\leq i\leq4$, where $a^2+a+1\neq0$. Then the polynomial
\begin{equation*}f(X)=X+aX^q+\gamma\Tr_q^{q^3}(c_1X+c_2X^2+c_3X^3+c_4X^{q+2})\end{equation*} is a PP of $\mathbb{F}_{q^3}$ if and only if one of the following holds:
 \begin{enumerate}
 \item[{\rm(i)}]  $\gamma(c_3+c_4)=0$, $\gamma c_2=0$ and $a+1+\gamma c_1\neq0$;
 \item[{\rm(ii)}]  $\gamma(c_3+c_4)=0$, $\gamma c_2\neq0$ and $a+1+\gamma c_1=0$;
  \item[{\rm(iii)}] $\gamma(c_3+c_4)\neq0$,  $\gamma(c_1c_3+c_1c_4+c_2^2)=(a+1)(c_3+c_4)$ and $m$ is odd.
 \end{enumerate}
\end{theorem}
\begin{proof}
 Let $\alpha\in\mathbb{F}_{q^3}^*$ with $\Tr_q^{q^3}(\alpha)=0$ such that $\{1,\alpha,\alpha^q\}$ is a basis of $\mathbb{F}_{q^3}$ over $\mathbb{F}_q$. Writing $X = x + y\alpha + z\alpha^q$ for $x, y, z \in \mathbb{F}_q$, we have
\begin{align*}
X^3&=x^3+(y\alpha+z\alpha^q)x^2+(y\alpha+z\alpha^q)^2x+(y\alpha+z\alpha^q)^3;\\
X^{q+2}&=(x+y\alpha^q+z\alpha^{q^2})(x+y\alpha+z\alpha^q)^2\\
  &=x^3+(y\alpha^q+z\alpha^{q^2})x^2+(y\alpha+z\alpha^q)^2x+(y\alpha^q+z\alpha^{q^2})(y\alpha+z\alpha^q)^2.
\end{align*}
Hence,
\begin{align*}
f(X)&=f(x+y\alpha+z\alpha^q)\\
  &=(a+1)x+(y+az)\alpha+(z+ay+az)\alpha^q+\gamma\left((c_3+c_4)x^3+c_2x^2+c_1x+\ell(y,z)\right),
\end{align*}
where $\ell(y,z)=\Tr_q^{q^3}\left(c_3(y\alpha+z\alpha^q)^3+c_4(y\alpha^q+z\alpha^{q^2})(y\alpha+z\alpha^q)^2\right)\in\mathbb{F}_q[y,z]$.
According to Lemma \ref{lem3-1}, $f(X)$ is a PP of $\mathbb{F}_{q^3}$ if and only if for every $y,z\in\mathbb{F}_q$, the polynomial
\begin{equation*}
g(x,y,z)-\gamma \ell(y,z)=(a+1)x+\gamma\left((c_3+c_4)x^3+c_2x^2+c_1x\right)\end{equation*}
permutes $\mathbb{F}_q$, since $\gamma \ell(y,z) \in \mathbb{F}_{q}[y,z]$ is independent of $x$. 
Next we divide the analysis into two cases.

{\bf Case 1:} Assume $\gamma(c_3+c_4)=0$. Then $g(x,y,z)-\gamma \ell(y,z)=\gamma c_2x^2+(a+1+\gamma c_1)x$. It is well known that such a quadratic polynomial permutes $\mathbb{F}_q$ if and only if
either $\gamma c_2=0$ and $a+1+\gamma c_1\neq0$, or $\gamma c_2\neq0$ and $a+1+\gamma c_1=0$.

{\bf Case 2:} Assume $\gamma(c_3+c_4)\neq0$. We apply Table 7.1 from \cite{Lidl} to write $g(x, y, z)-\gamma \ell(y, z)$ in its normalized form:
\begin{equation*}x^3+\left(\frac{c_2^2}{(c_3+c_4)^2}+\frac{c_1}{c_3+c_4}+\frac{a+1}{\gamma(c_3+c_4)}\right)x.\end{equation*}
Then this polynomial permutes $\mathbb{F}_{q}$ if and only if $q\not\equiv1\pmod3$ and  \begin{equation*}\gamma(c_1c_3+c_1c_4+c_2^2)=(a+1)(c_3+c_4),\end{equation*} where
the equation is a simplified form of
\begin{equation*}\frac{c_2^2}{(c_3+c_4)^2}+\frac{c_1}{c_3+c_4}+\frac{a+1}{\gamma(c_3+c_4)}=0.\end{equation*}
This completes the proof.
\end{proof}

\subsection{PPs of the form $X+\gamma\Tr_q^{q^3}(h(X))$ extended from Table \ref{table1} }

 Jiang, Li and Qu  studied the permutation properties of polynomials having $X+\gamma\Tr_q^{q^3}(h(X))$ with $\gamma=1$ in \cite[Section 4]{LikqQulj2026}, where $q=2^m$ and $h(X)$ is chosen as a binomial (see Table \ref{table1}). From this investigation, they obtained some new PPs over $\mathbb{F}_{q^3}$. In this part, we extend these findings to the more general form $X + \gamma \Tr_q^{q^3}(h(X))$ with $\gamma \in \mathbb{F}_q$. Our proofs are based on the multivariate method.  We first point out a fundamental fact that will be used throughout this subsection. If $\{\alpha_1,\dots,\alpha_n\}$ is a basis of $\mathbb{F}_{q^n}$ over $\mathbb{F}_{q}$, then the set $\{\alpha_1^p,\dots,\alpha_n^p\}$ is also a basis, where $p$ is the characteristic of $\mathbb{F}_q$.

 \begin{table}[htbp]
  \centering
  \caption{PPs of the form  $X+\Tr_q^{q^3}(h(X))$
  over $\mathbb{F}_{q^3}$ \\
  with $q=2^m$ and $h(X)$ as a binomial in \cite{LikqQulj2026}}\label{table1}
  \renewcommand{\arraystretch}{0.9}  
  \begin{tabular}{lll}
  \toprule
  {\rm No.} & $f(X)$ & {\rm Conditions} \\
  \hline
 1& $X+\Tr_q^{q^3}(X^{\frac{q^2+q}{2}}+X^{\frac{q^2+q+2}{2}})$  &{\rm all} $m$ \\
 2& $X+\Tr_q^{q^3}(X^{\frac{q^2+q}{2}}+X^{\frac{q^2-q+1}{2}})$  &{\rm all} $m$ \\
3& $X+\Tr_q^{q^3}(X^{\frac{q^2+q}{2}}+X^{2q+1})$  &$m$ {odd}\\
4& $X+\Tr_q^{q^3}(X^{2q+1}+X^{4q+1})$  &$m$ {odd}\\
5& $X+\Tr_q^{q^3}(X^{2q+2}+X^{4q+1})$  &$m\not\equiv0\pmod4$\\
6& $X+\Tr_q^{q^3}(X^{\frac{q^2+q}{2}}+X^{\frac{q^2+3q}{2}})$  &$m\not\equiv2\pmod3$ \\
7& $X+\Tr_q^{q^3}(X^{\frac{q^2+q}{2}}+X^{\frac{3q^2+q}{2}})$  &$m\not\equiv1\pmod3$ \\
\bottomrule
\end{tabular}
\end{table}

To facilitate the subsequent analysis, we first present the following Lemma.
\begin{lemma}\label{lem3-4}
Let $q=2^m$ and $\alpha\in\mathbb{F}_{q^3}^*$ such that $\Tr_q^{q^3}(\alpha)=0$ and $\{1,\alpha,\alpha^q\}$ is a basis of $\mathbb{F}_{q^3}$ over $\mathbb{F}_q$.  For any $X \in \mathbb{F}_{q^3}$, write $X=x+y\alpha+z\alpha^q$ with $x,y,z\in\mathbb{F}_q$. Then the following hold:
\begin{enumerate}
\item[{\rm(i)}]  $\Tr_q^{q^3}(X^{q+1})=x^2+(y^2+z^2+yz)\Tr_q^{q^3}(\alpha^{q+1})$;
\item[{\rm(ii)}] $\Tr_q^{q^3}(X^{q^2+q+2})=x^4+x^2(y^2+z^2+yz)\Tr_q^{q^3}(\alpha^{q+1})+x\Tr_q^{q^3}\left((y\alpha+z\alpha^q)^3\right)$;
\item[{\rm(iii)}] $\Tr_q^{q^3}(X^{2q+1})=x^3+\Tr_q^{q^3}\left((z\alpha+y\alpha^q+z\alpha^q)^2(y\alpha+z\alpha^q)\right)$;
\item[{\rm(iv)}]  $\Tr_q^{q^3}(X^{4q+1})=x^5+\Tr_q^{q^3}\left((z\alpha+y\alpha^q+z\alpha^q)^4(y\alpha+z\alpha^q)\right)$;
\item[{\rm(v)}]  $\Tr_q^{q^3}(X^{q+3})=x^4+x^2(y^2+z^2+yz)\Tr_q^{q^3}(\alpha^{q+1})+x\ell_{51}(y,z)+\ell_{52}(y,z)$;
\item[{\rm(vi)}]
 $\Tr_q^{q^3}(X^{3q+1})=x^4+x^2(y^2+z^2+yz)\Tr_q^{q^3}(\alpha^{q+1})+x\ell_{61}(y,z)+\ell_{62}(y,z)$;
\end{enumerate}
where $\ell_{51}(y,z)$, $\ell_{52}(y,z)$, $\ell_{61}(y,z)$  and $\ell_{62}(y,z)$ are constants in $\mathbb{F}_q$ with respect to $x$, given by 
\begin{equation*}\ell_{51}(y,z)=\Tr_q^{q^3}\left((y\alpha+z\alpha^q)^2(y\alpha^q+z\alpha^{q^2})+(y\alpha+z\alpha^q)^3\right);\end{equation*}
\begin{equation*}\ell_{52}(y,z)=\Tr_q^{q^3}\left((y\alpha^q+z\alpha^{q^2})(y\alpha+z\alpha^q)^3\right);\end{equation*}
\begin{equation*}\ell_{61}(y,z)=\Tr_q^{q^3}\left((y\alpha^q+z\alpha^{q^2})^2(y\alpha+z\alpha^q)+(y\alpha^q+z\alpha^{q^2})^3\right);\end{equation*}
\begin{equation*}\ell_{62}(y,z)=\Tr_q^{q^3}\left((y\alpha^q+z\alpha^{q^2})^3(y\alpha+z\alpha^q)\right).\end{equation*}
\end{lemma}\begin{proof}
The proofs of these cases are closely analogous. For brevity, we present the details only for the first two. By our choice of $\alpha$, every element $X\in\mathbb{F}_{q^3}$ can be uniquely expressed as $X=x+y\alpha+z\alpha^q$ with $x, y, z \in \mathbb{F}_q$. A direct computation shows:
\begin{equation*}\begin{aligned}
X^{q+1}&=(x+y\alpha^q+z\alpha^{q^2})(x+y\alpha+z\alpha^q)\\
 &=x^2+\left(y\alpha+y\alpha^q+z\alpha\right)x+(y\alpha+z\alpha^q)(y\alpha^q+z\alpha^q+z\alpha);\\
 &=x^2+\left(y\alpha+y\alpha^q+z\alpha\right)x+\left(yz\alpha^2+(y^2+z^2+yz)\alpha^{q+1}+(yz+z^2)\alpha^{2q}\right).
\end{aligned}\end{equation*}
\begin{equation}\begin{aligned}\label{X^q^2+q+1}
X^{q^2+q+1}&=(x+y\alpha+z\alpha^q)(x+y\alpha^q+z\alpha^{q^2})(x+y\alpha^{q^2}+z\alpha)\\
           &=x^3+x(y^2+z^2+yz)(\alpha^{2q}+\alpha^{q+1}+\alpha^{2})+(y\alpha+z\alpha^q)(y\alpha^q+z\alpha^{q^2})(y\alpha^{q^2}+z\alpha)\\
           &=x^3+x(y^2+z^2+yz)\Tr_q^{q^3}(\alpha^{q+1})+\Tr_q^{q^3}\left((y\alpha+z\alpha^q)^3\right).
\end{aligned}\end{equation}
Applying the trace function to both sides, and noting that $\Tr_q^{q^3}(\alpha) = 0$, it follows that
\begin{equation*}\begin{aligned}
\Tr_q^{q^3}(X^{q+1})=x^2+(y^2+z^2+yz)\Tr_q^{q^3}(\alpha^{q+1}).
\end{aligned}\end{equation*}
\begin{align*}
\Tr_q^{q^3}(X^{q^2+q+2})&=X^{q^2+q+1}\Tr_q^{q^3}(X)=x^4+x^2(y^2+z^2+yz)\Tr_q^{q^3}(\alpha^{q+1})+x\Tr_q^{q^3}\left((y\alpha+z\alpha^q)^3\right),
\end{align*}
where we use the fact that $X^{q^2+q+1}\in\mathbb{F}_q$( since $X^{(q^2+q+1)(q-1)}=X^{q^3-1}=1$).  We are done.
\end{proof}

It has been shown by Jiang, Li, and Qu that the first case listed in Table \ref{table1} is indeed a PP over $\mathbb{F}_{q^3}$. Specifically, they proved in \cite[Theorem 4.1]{LikqQulj2026} relying on the resultant of polynomials. It is worth noting that the same result can be derived more directly by applying Lemma \ref{lem3-1}. We slightly generalize this result and provide an alternative proof in the following.

\begin{theorem}\label{26-1}
Let $q=2^m$ and $\gamma\in\mathbb{F}_q$. Then $f(X)=X+\gamma\Tr_q^{q^3}(X^{\frac{q+1}{2}}+X^{\frac{q^2+q+2}{2}})$ is a PP of $\mathbb{F}_{q^3}$ if and only if $\gamma=0$ or $\gamma=1$.
\end{theorem}
\begin{proof}
It is trivial that if $\gamma=0$, then $f(X)=X$ is a PP of $\mathbb{F}_{q^3}$. Now assume $\gamma\neq0$. Observe that $f(X)$ has the same permutation behavior as \begin{equation*}F(X)=X^2+\gamma\Tr_q^{q^3}(X^{q+1}+X^{q^2+q+2}),\end{equation*}
so it suffices to study the permutation behavior of $F(X)$. Choose an element $\alpha\in\mathbb{F}_{q^3}^*$ such that $\Tr_q^{q^3}(\alpha)=0$ and $\{1,\alpha,\alpha^q\}$ forms a basis of $\mathbb{F}_{q^3}$ over $\mathbb{F}_q$.
Let $X=x+y\alpha+z\alpha^q$ with $x,y,z\in\mathbb{F}_q$. Then plugging the results from Lemma \ref{lem3-4}(i) and (ii) into $F(X)$, we obtain
\begin{align*}
F(X)&=x^2+y^2\alpha^2+z^2\alpha^{2q}+
\gamma\left(x^2+x^4+x^2(y^2+z^2+yz)\Tr_q^{q^3}(\alpha^{q+1})
 +x\Tr_q^{q^3}\left((y\alpha+z\alpha^q)^3\right)
 +\ell(y,z)\right),
\end{align*}
where $\ell(y,z)=(y^2+z^2+yz)\Tr_q^{q+1}(\alpha^{q+1})\in\mathbb{F}_q[y,z]$ is independent of $x$. By Lemmas \ref{constant} and \ref{lem3-1}, $F(X)$ is a PP of $\mathbb{F}_{q^3}$ if and only if for each $y,z\in\mathbb{F}_{q}$,
\begin{equation*}
g(x,y,z)=x^2+
\gamma\left(x^2+x^4+ x^2(y^2+z^2+yz)\Tr_q^{q^3}(\alpha^{q+1})
+x\Tr_q^{q^3}\left((y\alpha+z\alpha^q)^3\right)\right)
\end{equation*}
is a PP of $\mathbb{F}_q$. Let us now examine a special case by setting $y=z=0$. That is,
\begin{equation*}
g(x,0,0)=\gamma x^4+(\gamma+1)x^2. \end{equation*}
Clearly $g(x,0,0)$
permutes $\mathbb{F}_{q}$ if and only if $\gamma=1$, which implies that  for $\gamma\in\mathbb{F}_{q}^*\backslash\{1\}$, $F(X)$ cannot be a PP of $\mathbb{F}_{q^3}$.

It remains to prove that $F(X)$ is indeed a PP of $\mathbb{F}_{q^3}$ for $\gamma=1$. As indicated by the above analysis, it is suffices to show that for any $y,z\in\mathbb{F}_q$, the linearized polynomial
\begin{equation*}g(x,y,z)=x^4+ x^2(y^2+z^2+yz)\Tr_q^{q^3}(\alpha^{q+1})
+x\Tr_q^{q^3}\left((y\alpha+z\alpha^q)^3\right)\end{equation*}
is a PP of $\mathbb{F}_q$. That is to say, for any $y,z\in\mathbb{F}_q$,
\begin{equation*}x^3+x(y^2+z^2+yz)\Tr_q^{q^3}(\alpha^{q+1})+\Tr_q^{q^3}\left((y\alpha+z\alpha^q)^3\right)=0\end{equation*}
has no solution in $\mathbb{F}_q$.
Observing the expression (\ref{X^q^2+q+1}) in the proof of Lemma \ref{lem3-4}, we see that this cubic equation can be precisely factored as 
\begin{equation*}(x+y\alpha+z\alpha^q)(x+y\alpha^q+z\alpha^{q^2})(x+y\alpha^{q^2}+z\alpha)=0.\end{equation*}
Hence, the three roots of the above equation are $y\alpha+z\alpha^q$, $y\alpha^q+z\alpha^{q^2}$ and $y\alpha^{q^2}+z\alpha$. Clearly, none of these roots lies in $\mathbb{F}_q$, thus completing the proof.
\end{proof}

The following corollary readily follows from the proof of Theorem \ref{26-1}.
\begin{corollary}
Let $q=2^m$. Then for any positive integer $j$, the polynomial \begin{equation*}f(X)=X+\Tr_q^{q^3}(X^{\frac{q+1}{2}}+X^{\frac{(q^2+q+2)2^j}{2}})\end{equation*} is a PP of $\mathbb{F}_{q^3}$.
\end{corollary}

\begin{theorem}
Let $q=2^m$  and $\gamma\in\mathbb{F}_q$. Then $f(X)=X+\gamma\Tr_q^{q^3}(X^{\frac{q+1}{2}}+X^{\frac{q^2-q+1}{2}})$ is a PP of $\mathbb{F}_{q^3}$ if and only if $\gamma=0$ or $\gamma=1$.
\end{theorem}
\begin{proof}
Clearly $f(X)=X$ is a PP of $\mathbb{F}_{q^3}$. Assume $\gamma\neq0$ in the following. It suffices to study the permutation behavior of $F(X)=f(X^2)$ over $\mathbb{F}_{q^3}$. By Corollary \ref{basis}, there exists $\alpha\in\mathbb{F}_{q^3}^*$ satisfying  $\Tr_q^{q^3}(\alpha)=0$ and $\{1,\alpha,\alpha^q\}$ forms a basis of $\mathbb{F}_{q^3}$ over $\mathbb{F}_q$.
Let $X=x+y\alpha+z\alpha^q$ with $x,y,z\in\mathbb{F}_q$. Note that
$X^{q^2-q+1}=X^{2q^2+2}/X^{q^2+q+1}$ and $X^{q^2+q+1}\in\mathbb{F}_q$. We have
\begin{equation*}\Tr_q^{q^3}(X^{q^2-q+1})=\Tr_q^{q^3}\left(\frac{X^{2q^2+2}}{X^{q^2+q+1}}\right)=\frac{\left(\Tr_q^{q^3}(X^{q+1})\right)^2}{X^{q^2+q+1}}.\end{equation*}
By Lemma \ref{lem3-4}, $F(X)$ can be rewritten as
\begin{align*}
F(X)&=X^2+\gamma\Tr_q^{q^3}(X^{q+1}+X^{q^2-q+1})\\
&=x^2+y^2\alpha^2+z^2\alpha^{2q}+\gamma\left(x^2+\ell(y,z)+\frac{x^4+\ell^2(y,z)}{X^{q^2+q+1}}\right)\\
   &=x^2+y^2\alpha^2+z^2\alpha^{2q}+\gamma\left(x^2+\ell(y,z)+\frac{x^4+\ell^2(y,z)}{x^3+x(y^2+z^2+yz)\Tr_q^{q^3}(\alpha^{q+1})+\Tr_q^{q^3}\left((y\alpha+z\alpha^q)^3\right)}\right),\\
\end{align*}
where $\ell(y,z)=(y^2+z^2+yz)\Tr_q^{q+1}(\alpha^{q+1})\in\mathbb{F}_q[y,z]$ is independent of $x$.
By Lemmas \ref{constant} and \ref{lem3-1}, $F(X)$ is a PP of $\mathbb{F}_{q^3}$ if and only if for every $y,z\in\mathbb{F}_q$,
\begin{equation*}g(x,y,z)=x^2+\gamma\left(x^2+\frac{x^4+\ell^2(y,z)}{x^3+x(y^2+z^2+yz)\Tr_q^{q^3}(\alpha^{q+1})+\Tr_q^{q^3}\left((y\alpha+z\alpha^q)^3\right)}\right)\end{equation*}
is a PP of $\mathbb{F}_{q}$. In the special case $y=z=0$, we obtain $g(x,0,0)=x^2+\gamma(x^2+x)=(\gamma +1)x^2+\gamma x$. It is easy to see that for $\gamma\neq0$, this quadratic polynomial is a PP of $\mathbb{F}_q$ only when $\gamma=1$. Therefore, $F(X)$ cannot be a PP of $\mathbb{F}_{q^3}$ for $\gamma\in\mathbb{F}_{q}^*\backslash\{1\}$.
Finally, we conclude this result by showing that $F(X)$ is a PP of $\mathbb{F}_{q^3}$ for $\gamma = 1$. This  result follows directly from \cite[Theorem 4.2]{LikqQulj2026}.
\end{proof}

\begin{theorem}
Let $q=2^m$, $\gamma\in\mathbb{F}_q$ and $j\in\{0,1\}$. Then $f(X)=X+\gamma\Tr_q^{q^3}\left(X^{\frac{q+1}{2}}+X^s\right)$, where $s=2q+1$ or $s=\frac{q^3+2q}{2}$ is a PP of $\mathbb{F}_{q^3}$ if and only if $\gamma=0$; or $\gamma=1$ and $m$ is odd.
\end{theorem}
\begin{proof} It suffices to study the permutation behavior of $F(X)=f(X^2)$. We only prove for $\gamma\neq0$. Let $\alpha\in\mathbb{F}_{q^3}^*$ with $\Tr_q^{q^3}(\alpha)=0$ such that $\{1,\alpha,\alpha^q\}$ is a basis of $\mathbb{F}_{q^3}$ over $\mathbb{F}_q$. Let $X=x+y\alpha+z\alpha^q$ with $x,y,z\in\mathbb{F}_q$. According to Lemma \ref{lem3-4} (i) and (iii), we can reformulate $F(X)$ as follows:
\begin{align*}
F(X)&=X^2+\gamma\Tr_q^{q^3}(X^{q+1}+X^{(2q+1)2^j})\\
&=x^2+y^2\alpha^2+z^2\alpha^{2q}+\gamma\left(x^2+x^{3\cdot2^j}+\ell(y,z)\right),
\end{align*}
where the parameter $j$ depends on $s$: $j=1$ corresponds to $s=2q+1$, $j=0$ corresponds to $s=\frac{q^3+2q}{2}$, and $\ell(y,z)=(y^2+z^2+yz)\Tr_q^{q^3}(\alpha^{q+1})+ \Tr_q^{q^3}\left((z\alpha+y\alpha^q+z\alpha^q)^{2^{j+1}}(y\alpha+z\alpha^q)^{2^j}\right)\in\mathbb{F}_{q}[y,z]$ is a constant in $\mathbb{F}_q$ with respect to $x$. It follows that $F(X)$ is a PP of $\mathbb{F}_{q^3}$ if and only if for every $y,z\in\mathbb{F}_q$,
\begin{equation*}g(x,y,z)=x^2+\gamma\left(x^2+x^{3\cdot2^j}+\ell(y,z)\right),\end{equation*}
i.e., \qquad \qquad \qquad \qquad \qquad \quad
$\left(g(x,y,z)-\gamma\ell(y,z)\right)\gamma^{-1}=x^{3\cdot2^j}+(1+\gamma^{-1})x^2$

\noindent is a PP of $\mathbb{F}_q$. For the case $j=1$, we observe that $x^6+(1+\gamma^{-1})x^2=\left(x^3+(1+\gamma^{-1})x\right)\circ x^2$. The permutation property of the left-hand side is therefore equivalent to that of $x^3+(1+\gamma^{-1})x$. According to Table 7.1 in \cite{Lidl}, $x^3+(1+\gamma^{-1})x$ permutes $\mathbb{F}_{q}$ if and only if $q\not\equiv1\pmod3$ and $\gamma=1$. For the case $j=0$, the polynomial reduces to $x^3+(1+\gamma^{-1})x^2$, whose permutation property is equivalent to its normalized form $x^3+(1+\gamma^{-1})^2x$. Applying  Table 7.1 in \cite{Lidl} once more yields $x^3+(1+\gamma^{-1})^2x$ permutes $\mathbb{F}_q$ if and only if $q\not\equiv1\pmod3$ and $\gamma=1$. Given $q=2^m$, the condition $q\not\equiv1\pmod3$ is equivalent to $m$ being odd. Consequently, both cases require $m$ be odd and $\gamma=1$. We thus get the desired result.
\end{proof}

\begin{theorem}
Let $q=2^m$  and $\gamma\in\mathbb{F}_q$. Then $f(X)=X+\gamma\Tr_q^{q^3}(X^{2q+1}+X^{4q+1})$ is a PP of $\mathbb{F}_{q^3}$ if and only if $\gamma=0$; or $\gamma=1$ and $m$ is odd.
\end{theorem}
\begin{proof}
We only prove for $\gamma\neq0$. Let $\alpha\in\mathbb{F}_{q^3}^*$ with $\Tr_q^{q^3}(\alpha)=0$ such that $\{1,\alpha,\alpha^q\}$ is a basis of $\mathbb{F}_{q^3}$ over $\mathbb{F}_q$. Let $X=x+y\alpha+z\alpha^q$ with $x,y,z\in\mathbb{F}_q$. Substituting Lemma \ref{lem3-4} (iii)-(iv) into $f(X)$ yields that \begin{align*}
f(X)=x+y\alpha+z\alpha^{q}+\gamma\left(x^3+x^5+\ell(y,z)\right),
\end{align*}
where $\ell(y,z)=\Tr_q^{q^3}((z\alpha+y\alpha^q+z\alpha^q)^2(y\alpha+z\alpha^q))+\Tr_q^{q^3}((z\alpha+y\alpha^q+z\alpha^q)^4(y\alpha+z\alpha^q))\in\mathbb{F}_q[y,z]$.
Now by Lemma \ref{lem3-1}, $f(X)$ is a PP of $\mathbb{F}_{q^3}$ if and only if, for each $y,z\in\mathbb{F}_q$,
\begin{equation*}g(x,y,z)=x+\gamma(x^5+x^3+\ell(y,z))\end{equation*}
is  a PP of $\mathbb{F}_q$. Given $\gamma\ne0$ and fixed $y,z\in\mathbb{F}_q$, it suffices to analyze the permutation property of $g'(x)=x^5+x^3+\gamma^{-1}x$. From Table 7.1 in \cite{Lidl}, we assert that $g'(x)$ is a PP of $\mathbb{F}_q$ if and only if $q\equiv\pm2\pmod5$ and $\gamma=5$. That is $m$ being odd and $\gamma=1$.
\end{proof}

\begin{theorem}
Let $q=2^m$  and $\gamma\in\mathbb{F}_q$. Then $f(X)=X+\gamma\Tr_q^{q^3}(X^{2q+2}+X^{4q+1})$ is a PP of $\mathbb{F}_{q^3}$ if and only if $\gamma=0$; or $\gamma=1$ and $4\nmid m$.
\end{theorem}
\begin{proof} We only prove for $\gamma\neq0$. Let $\alpha\in\mathbb{F}_{q^3}^*$ with $\Tr_q^{q^3}(\alpha)=0$ such that $\{1,\alpha,\alpha^q\}$ is a basis of $\mathbb{F}_{q^3}$ over $\mathbb{F}_q$. Let $X=x+y\alpha+z\alpha^q$ with $x,y,z\in\mathbb{F}_q$. We have
\begin{align*}f(X)=&X+\gamma\left(\left(\Tr_q^{q^3}(X^{q+1})\right)^2+\Tr_q^{q^3}(X^{4q+1})\right)\\
&=x+y\alpha+z\alpha^q+\gamma(x^5+x^4+\ell(y,z)),\end{align*}
where $\ell(y,z)=\left((y^2+z^2+yz)\Tr_q^{q^3}(\alpha^{q+1})\right)^2+\Tr_q^{q^3}\left((z\alpha+y\alpha^q+z\alpha^q)^4(y\alpha+z\alpha^q)\right)\in\mathbb{F}_q[y,z]$. By Lemma \ref{lem3-1}, the problem now reduces to analyzing the permutation behaviour of
\begin{equation*}g(x,y,z)=x+\gamma\left(x^5+x^4+\ell(y,z)\right),\end{equation*}
over $\mathbb{F}_q$, where $y$ and $z$ run through $\mathbb{F}_q$. Rewriting this expression, we obtain
\begin{equation*}\left(g(x,y,z)-\gamma\ell(y,z)\right)\gamma^{-1}=x^5+x^4+\gamma^{-1}x,\end{equation*}
 whose normalized form is $x^5+(1+\gamma^{-1})x$. Clearly by Table 7.1 in \cite{Lidl}, this polynomial is a PP of $\mathbb{F}_q$ if and only if $\gamma=1$ and $q\not\equiv1\pmod5$, i.e., $\gamma=1$ and $4\nmid m$. Therefore, $f(X)$ is a PP of $\mathbb{F}_{q^3}$ if and only if  $\gamma=1$ and $4\nmid m$.
\end{proof}

\begin{theorem}\label{th3.9}
 Let $q=2^m$ with $m\not\equiv2\pmod3$ and $\gamma\in\mathbb{F}_q$.  Then $f(X)=X+\gamma\Tr_q^{q^3}(X^{\frac{q+1}{2}}+X^{\frac{q+3}{2}})$ is a PP of $\mathbb{F}_{q^3}$ if and only if $\gamma=0$ or $\gamma=1$.
\end{theorem}
\begin{proof} It suffices to study the permutation behavior of $F(X)=f(X^2)$. We only prove for $\gamma\neq0$. Let $\alpha\in\mathbb{F}_{q^3}^*$ with $\Tr_q^{q^3}(\alpha)=0$ such that $\{1,\alpha,\alpha^q\}$ is a basis of $\mathbb{F}_{q^3}$ over $\mathbb{F}_q$. Let $X=x+y\alpha+z\alpha^q$ with $x,y,z\in\mathbb{F}_q$. By Lemma \ref{lem3-4}, $F(X)$ can be be expressed as
\begin{align*}
F(X)&=X^2+\gamma\Tr_q^{q^3}(X^{q+1}+X^{q+3})\\
&=x^2+y^2\alpha^2+z^2\alpha^{2q}+\gamma\left(x^2+x^4+x^2(y^2+z^2+yz)\Tr_q^{q^3}(\alpha^{q+1})+x\ell_{51}(y,z)+\ell(y,z)\right),\end{align*}
where $\ell(y,z)=(y^2+z^2+yz)\Tr_{q}^{q^3}(\alpha^{q+1})+\ell_{52}(y,z)\in\mathbb{F}_{q}[y,z]$, and $\ell_{51}(y,z)$, $\ell_{52}(y,z)$ are as defined in Lemma \ref{lem3-4}.  Applying Lemma \ref{lem3-1}, $F(X)$ is a PP of $\mathbb{F}_q$ if and only if for every $y,z\in\mathbb{F}_q$,
\begin{equation*}g(x,y,z)=x^2+\gamma\left(x^2+x^4+x^2(y^2+z^2+yz)\Tr_q^{q^3}(\alpha^{q+1})+x\ell_{51}(y,z)+\ell(y,z)\right)\end{equation*}
is a PP of $\mathbb{F}_q$. We now consider the case where $y=z=0$. Given $\gamma\neq0$, it is clear that $g(x,0,0)=\gamma x^4+(\gamma+1)x^2$ permutes  $\mathbb{F}_q$ if and only if $\gamma=1$.
 It remains to prove that for $\gamma=1$, $F(X)$ is indeed a PP of $\mathbb{F}_q$. According to \cite[Theorem 4.3]{LikqQulj2026}, we complete the proof.
\end{proof}

\begin{theorem}\label{26-Li}
Let $q=2^m$ with $m\not\equiv1\pmod3$ and $\gamma\in\mathbb{F}_q$.  Then $f(X)=X+\gamma\Tr_q^{q^3}(X^{\frac{q+1}{2}}+X^{\frac{3q+1}{2}})$ is a PP of $\mathbb{F}_{q^3}$ if and only if $\gamma=0$ or $\gamma=1$.
\end{theorem}
\begin{proof}
By a similar argument to the proof of Theorem \ref{th3.9}, the result follows; we omit the details here.
\end{proof}

\begin{remark}
Theorems \ref{26-1}-\ref{26-Li} show the polynomials listed in Table \ref{table1}, which are of the form  $X+\gamma\Tr_q^{q^3}(h(X))$ with $\gamma=1$,  cannot be PPs for any $\gamma\in\mathbb{F}_q^*\backslash\{1\}$.
\end{remark}

Subsequently, we aim to characterize further families of permutation polynomials over $\mathbb{F}_{q^3}$.

\begin{theorem}
Let $q=2^m$ and $\gamma\in\mathbb{F}_q$. Then the polynomial \begin{equation*}f(X)=X+\gamma\Tr_q^{q^3}(X^{\frac{q+1}{2}}+X^{\frac{q^2+q+2}{2}}+X^{\frac{q+3}{2}}+X^{\frac{3q+1}{2}})\end{equation*}
is a PP of $\mathbb{F}_{q^3}$ if and only if $\gamma=0$.
\end{theorem}
\begin{proof}
It suffices to study the permutation behavior of $F(X)=f(X^2)$. We only prove the case for $\gamma \neq 0$. Choose  an element $\alpha \in \mathbb{F}_{q^3}^*$ with $\Tr_q^{q^3}(\alpha) = 0$ such that $\{1, \alpha, \alpha^q\}$ forms a basis of $\mathbb{F}_{q^3}$ over $\mathbb{F}_q$. Express $X = x + y\alpha + z\alpha^q$ with $x, y, z \in \mathbb{F}_q$. By Lemma \ref{lem3-4}, we have
\begin{align*}
F(X)&=X^2+\gamma\Tr_q^{q^3}(X^{q+1}+X^{q^2+q+2}+X^{q+3}+X^{3q+1})\\
&= x^2 + y^2\alpha^2 + z^2\alpha^{2q}+\gamma\left(x^4+x^2+x^2(y^2+z^2+yz)\Tr_q^{q^3}(\alpha^{q+1})\right)\\
      &~~~~
  +\gamma x\left(\Tr_q^{q^3}((y\alpha+z\alpha^q)^3)+\ell_{51}(y,z)+\ell_{61}(y,z)\right)+\gamma\ell(y,z),
\end{align*}
where $\ell(y,z)=(y^2+z^2+yz)\Tr_q^{q^3}(\alpha^{q+1})+\ell_{52}(y,z)+\ell_{62}(y,z)\in\mathbb{F}_q[y,z]$ is a constant with respect to $x$. By Lemmas \ref{constant} and \ref{lem3-1}, $F(X)$ is a PP of $\mathbb{F}_{q^3}$ if and only if for each $y,z\in\mathbb{F}_q$,
\begin{equation*}g(x,y,z)=x^2+\gamma\left(x^4+x^2+x^2(y^2+z^2+yz)\Tr_q^{q^3}(\alpha^{q+1})+x\left(\Tr_q^{q^3}((y\alpha+z\alpha^q)^3)+\ell_{51}(y,z)+\ell_{61}(y,z)\right)\right)\end{equation*}
is a PP of $\mathbb{F}_q$.  As a special case, setting $y=z=0$ gives
\begin{equation*}g(x,0,0)=x^2+\gamma(x^4+x^2),\end{equation*}
whose normalized form is $x^4+(1+\gamma^{-1})x^2$ as $\gamma\neq0$. It follows that $g(x,0,0)$ permutes $\mathbb{F}_q$ if and only if $\gamma=1$. This observation allows us to conclude that $F(X)$ cannot be a PP of $\mathbb{F}_{q^3}$ for any $\gamma \in \mathbb{F}_q^* \backslash\{1\}$.

We now consider the remaining case where $\gamma=1$.
 It suffices to analyze the number of solutions to the equation
\begin{equation}\label{eq8}
X^2+\Tr_q^{q^3}(X^{q+1}+X^{q^2+q+2}+X^{q+3}+X^{3q+1})=a^2\end{equation}
Substituting $X=a+u$ into (\ref{eq8}) yields
\begin{equation*}
u^4+\Tr_q^{q^3}(a^{q+1}+a^{2q})u^2+\Tr_q^{q^3}(a^{q+1}+a^{q^2+q+2}+a^{q+3}+a^{3q+1})=0,\end{equation*}
or equivalently,
\begin{equation*}
\left(u^2+\Tr_q^{q^3}(a^{q+1}+a^{2q})u+\Tr_q^{q^3}(a^{q+1}+a^{q^2+q+2}+a^{q+3}+a^{3q+1})\right)\circ u^2=0.\end{equation*}
Therefore, $F(X)$ with $\gamma=1$ is a PP of $\mathbb{F}_{q^3}$ if and only if for each $a\in\mathbb{F}_{q^3}$,
\begin{equation}\label{eq3} u^2+Au+B=0\end{equation}
has a unique solution in $\mathbb{F}_q$, where $A=\Tr_q^{q^3}(a^{q+1}+a^{2q})$ and $B=\Tr_q^{q^3}(a^{q+1}+a^{q^2+q+2}+a^{q+3}+a^{3q+1})$.
However, we note that for any $a\in\mathbb{F}_{q^3}$ such that $A\neq0$, Equation (\ref{eq3}) has either exactly two distinct roots or no roots in $\mathbb{F}_q$. Therefore, $F(X)$ with $\gamma=1$ cannot be a permutation of $\mathbb{F}_{q^3}$. We complete the proof.
\end{proof}

\section{Conclusion}

The objective of this paper is to provide a novel approach for constructing PPs of the form $L(X)+\gamma\Tr_q^{q^3}(h(X))$ over finite fields. The main idea of this paper is to reformulate the challenge of constructing univariate permutation polynomials over finite fields as the task of constructing multivariate permutations over $\mathbb{F}_{q}$-vector spaces, thereby offering a new pathway for analysis. It is worth noting that several classes of PPs discussed in this work are inspired by the work of \cite{LikqQulj2025, LikqQulj2026}.

Although our current study focuses on  $\mathbb{F}_{q^3}$ with even characteristic, Lemmas \ref{basisfromNB} and  \ref{pang-yuan-wu-guan} are established for the more general case. This naturally provides a promising direction for future research: constructing new PPs in finite fields of higher degree or odd characteristic by leveraging the techniques developed in this paper.


\textit{ \begin{funding}{\rm Pingzhi Yuan was supported by the National Natural Science Foundation of China (Grant No. 12171163, No. 12571003).
Danyao Wu was supported by the National Natural Science Foundation of China (Grants No. 12501006).}
\end{funding}}

\section*{Declarations}
\begin{conflict of interest} {\rm There is no conflict of interest.}
\end{conflict of interest}



\end{document}